\RequirePackage{fix-cm}
\documentclass[smallextended]{svjour3}       
\smartqed  
\usepackage{graphicx}
\usepackage{amsmath,amsfonts}
\usepackage[hidelinks]{hyperref}
\usepackage{tikz-cd}
\usepackage{stmaryrd}
\SetSymbolFont{stmry}{bold}{U}{stmry}{m}{n}
\numberwithin{equation}{section}
\numberwithin{lemma}{section}
\numberwithin{remark}{section}
\usepackage{bm}

\newcommand{\curl}{\nabla \times}
\renewcommand{\div}{\nabla \cdot}
\newcommand{\rot}{\nabla^\perp}
\DeclareMathOperator{\Ker}{Ker}
\DeclareMathOperator{\Ran}{Ran}

\renewcommand{\d}{\mathrm{d}}

\journalname{Journal}
\begin{document}

\title{A Reduced Basis Method for Darcy flow systems that ensures local mass conservation by using exact discrete complexes
\thanks{This project has received funding from the European Union's Horizon 2020 research and innovation programme under the Marie Skłodowska-Curie grant agreement No. 101031434 -- MiDiROM.}
}

\titlerunning{A Reduced Basis Method that ensures local mass conservation}        

\author{Wietse M. Boon \and Alessio Fumagalli}

\institute{W.M. Boon \at
              MOX -- Modelling and Scientific Computing, Politecnico di Milano, Milano, Italy
              \email{wietsemarijn.boon@polimi.it}
           \and
           A. Fumagalli \at
              MOX -- Modelling and Scientific Computing, Politecnico di Milano, Milano, Italy
              \email{alessio.fumagalli@polimi.it}
}

\date{Received: date / Accepted: date}

\maketitle

\begin{abstract}
    A solution technique is proposed for flows in porous media that guarantees local conservation of mass.
    We first compute a flux field to balance the mass source and then exploit exact co-chain complexes to generate a solenoidal correction. A reduced basis method based on proper orthogonal decomposition is employed to construct the correction and we show that mass balance is ensured regardless of the quality of the reduced basis approximation.
    The method is directly applicable to mixed finite and virtual element methods, among other structure-preserving discretization techniques, and we present the extension to Darcy flow in fractured porous media.
\end{abstract}

\section{Introduction}
\label{sec: Introduction}

The construction of inexact solution schemes involves deciding which errors are acceptable and which approximations can be made for the sake of computational efficiency. Herein, we consider the mixed formulation of Darcy flow systems and take the perspective that the physical law of mass conservation is significantly more important than the constitutive relationship known as Darcy's law. In line with this perspective, our goal is to formulate an efficient solution technique that guarantees local mass conservation.

Efficient solvers are of paramount importance in applications where multiple model realizations are necessary. In the context of uncertainty quantification, inverse modeling, or system optimization, for example, it is vital to understand the dependency of the solution on model parameters. However, obtaining this relationship typically requires solving a high-fidelity model multiple times, which can be computationally expensive, if not prohibitive.
This cost can be relieved by using Reduced Order Modeling (ROM) techniques in which the original problem is replaced by a model of lower numerical complexity.

The literature on ROM is vast and we refer the interested reader to \cite{hesthaven2016certified,quarteroni2015reduced,quarteroni2014reduced} and references therein. In this work, we focus on the Reduced Basis Methods (RBM) constructed using Proper Orthogonal Decomposition (POD) \cite[Sec. 6.3]{quarteroni2015reduced}. In particular, we use the well-established snapshot method, which originated from turbulent flow models \cite{sirovich1987turbulence}.

A direct application of RBM to the mixed formulation of Darcy flow would introduce an error in both the mass balance and constitutive equations. Hence, mass conservation cannot be guaranteed by the reduced basis solution. An alternative approach is to construct separate reduced bases for the flux and pressure variables. However, special considerations are then necessary to ensure inf-sup stability. Finally, if we were to correct an RBM solution through the use of a projection, then the projection operator can be as computationally expensive as the original problem. A more sophisticated approach is therefore required.

In this work, we propose a three-step solution procedure.
In the first step, an initial flux field is obtained by using a locally conservative method, such as the Finite Volume Method with a Two-Point Flux Approximation (TPFA). Although the TPFA scheme is computationally efficient, it generally lacks consistency and therefore requires a suitable correction, which is constructed in the second step. Since the mass balance is already satisfied at this stage, the correction needs to be divergence-free. The Helmholtz decomposition then ensures us that this correction can be described as the curl of a potential field $r$. The second step therefore employs an $H(\curl)$-conforming discretization to compute the potential and then updates the flux field with $\curl r$. Finally, the pressure field is constructed in the third step.

We restrict ourselves to discretization methods capable of conserving mass locally by which we mean that: (1) the mass is balanced in each element and (2) the normal flux is uniquely defined on each face of the mesh. Second, since our approach relies on fundamental properties such as the Helmholtz decomposition, we focus on structure-preserving methods, i.e. methods based on discrete spaces that form an exact discrete co-chain complex. Two discretization methods with these properties are used as leading examples, namely mixed finite element methods \cite{arnold2006finite,arnold2018finite} and mixed virtual element methods \cite{beirao2013basic}.

We then introduce the Reduced Basis Method in the second step to rapidly produce the potential field $r$ for a given conductivity distribution. Since the correction $\nabla \times r$ is guaranteed to be solenoidal, we ensure that it does not impact the mass conservation equation.

The procedure is presented in the general context of exact complexes. This allows us to directly extend the method to similar systems of equations, including Darcy flow systems in fractured porous media. By rewriting the equations in terms of the mixed-dimensional divergence, the problem can be identified as a mixed-dimensional Darcy flow system \cite{boon2018robust}. In turn, the solution procedure directly applies. In this case, we employ the mixed-dimensional curl to ensure that the correction step does not impact the mass conservation equation in the bulk, fractures, and fracture intersections. The mixed-dimensional curl, defined in \cite{boon2021functional}, has been used before in the analysis of mixed finite elements for elasticity \cite{boon2021stable} and in the construction of auxiliary space preconditioners \cite{budisa2020mixed}.

There are several similarities with the framework of auxiliary space preconditioning \cite{hiptmair2007nodal}. In particular, we use an exact complex to decompose the solution into an irrotational and a solenoidal part. However, we do not form a decomposition of higher regularity than $H(\curl)$ and thus directly work with edge-based instead of nodal elements. Moreover, our focus is on ROM rather than preconditioning.

We note that it is common to use the curl for generating solenoidal fields in the construction of stable finite element pairs for Stokes flow \cite{falk2013stokes,neilan2015discrete}. An important difference with our work is that we completely transfer the problem to the Sobolev space $H(\curl)$ instead of enriching the finite element spaces with additional basis functions.

In short, the main contributions of this work are:
\begin{itemize}
    \item A novel procedure is proposed that solves the mixed formulation of Darcy flow systems in three steps. In our example case of lowest-order, this procedure combines the efficiency of TPFA with the consistency of mixed finite element methods.
    \item We augment the procedure to obtain a Reduced Basis Method. The quality of the reduced order approximation does not affect the local conservation of mass.
    \item By presenting the method in an abstract setting, the extension to Darcy flow in fractured porous media follows immediately.
    \item The validity of the approach is confirmed by numerical experiments for cases in two and three dimensions, with immersed fracture networks.
\end{itemize}

The article is organized as follows. First, the model problem and our notation conventions are introduced in Sections~\ref{sub: model problem} and \ref{sub: preliminaries and notation}, respectively. Afterward, Section~\ref{sec: solution technique} presents the three-step procedure as it applies to Darcy flow in 3D and its generalization to the abstract setting of exact complexes. This section moreover shows the applicability to structure-preserving discretization methods and Darcy flow in fractured porous media. Section~\ref{sec: RBM} concerns the reduced basis method and its construction by proper orthogonal decomposition. The numerical implementation is discussed in Section~\ref{sec: Numerics} and we present experiments showing the performance of the method. Finally, Section~\ref{sec: Conclusions} contains the concluding remarks.

\subsection{The model problem}
\label{sub: model problem}

Let $\Omega \subset \mathbb{R}^n$ with $n \in \{2, 3\}$ be a contractible, bounded Lipschitz domain. Let the hydraulic conductivity $K$ be a symmetric, positive definite tensor field on $\Omega$ and let $f$ be the mass source. We then consider the Darcy flow problem: Find the pair $(q, p)$ such that
\begin{subequations} \label{eqs: Darcy}
\begin{align}
    q + K \nabla p &= 0, \label{eq: Darcy law} \\
    \div q &= f, \label{eq: mass balance}
    &\text{on }&\Omega,
\end{align}
subject to the boundary conditions
\begin{align}
    \nu \cdot q &= 0, &\text{on }&\partial_q \Omega, \\
    p &= g, &\text{on }&\partial_p \Omega,
\end{align}
\end{subequations}
with $\partial \Omega = \partial_q \Omega \cup \partial_p \Omega$ disjointly and $\nu$ the unit vector that is normal to $\partial \Omega$. We refer to problem \eqref{eqs: Darcy} as the \emph{Neumann} problem if $\partial \Omega = \partial_q \Omega$, the \emph{Dirichlet} problem if $\partial \Omega = \partial_p \Omega$, and \emph{mixed} otherwise.

\subsection{Preliminaries and notation}
\label{sub: preliminaries and notation}

The following notation is used throughout this work. First, let $L^2$ be the space of square integrable functions on $\Omega$ and let $\langle \cdot, \cdot \rangle$ denote the corresponding inner product. We reuse this notation for the inner product between vector-valued, square integrable functions. On the other hand, the notation with round brackets $(\cdot, \cdot)$ is reserved for tuples.

Let $(\nabla)$, $(\curl)$, and $(\div)$ denote the gradient, curl, and divergence operators, respectively. These differential operators induce the following Sobolev spaces:
\begin{align*}
    H(\div) &:= \{q \in (L^2)^n : \div q \in L^2 \}, & n &\in \{2,3\}, \\
    H(\curl) &:= \{r \in (L^2)^n : \curl r \in (L^2)^n \}, & n &= 3, \\
    H(\nabla) &:= \{s \in L^2 : \nabla s \in (L^2)^n \},  & n &\in \{2,3\}.
\end{align*}
We remark that $H(\nabla)$ is typically denoted by $H^1(\Omega)$ but we retain this notation for consistency.
Additionally, for $n = 2$, we define the rotated gradient $\rot := [-\partial_y, \partial_x]^T$ and $H(\rot) := H(\nabla)$.

Let the subspaces containing homogeneous boundary conditions on $\partial_q \Omega$ be denoted by
\begin{align*}
    H_{\partial_q \Omega}(\div) &:= \{q \in H(\div) : \nu \cdot q|_{\partial_q \Omega} = 0 \}, \\
    H_{\partial_q \Omega}(\curl) &:= \{r \in H(\curl) : \nu \times r|_{\partial_q \Omega} = 0 \}, \\
    H_{\partial_q \Omega}(\nabla) &:= \{s \in H(\nabla) : s|_{\partial_q \Omega} = 0 \}.
\end{align*}
We use the short-hand notation $H_0(\cdot)$ for $H_{\partial \Omega}(\cdot)$.
For an operator $\d$, we let $\Ran(\d)$ and $\Ker(\d)$ denote its range and kernel, respectively.

Finally, we use the Gothic font to denote mixed-dimensional entities, e.g. $(\mathfrak{q, p})$ introduced in Section~\ref{sub: MiDi}. The Sans Serif font is used to denote matrices and vectors, i.e. $\mathsf{Ax = b}$.

\section{A solution technique based on exact complexes}
\label{sec: solution technique}

We present a solution technique in which we first solve the mass balance equation \eqref{eq: mass balance}. We then exploit the exact de Rham complex to construct a solenoidal correction such that \eqref{eq: Darcy law} is satisfied as well. For ease of exposition, we first consider the three-dimensional Neumann problem in Section~\ref{sub: Neumann 3D} and present the general setting in Section~\ref{sub: general}. The discrete case is discussed in Section~\ref{sub: Discretization} and we show that the procedure can be applied to flows in fractured porous media in Section~\ref{sub: MiDi}.

\subsection{The Neumann problem in 3D}
\label{sub: Neumann 3D}

We start by considering the Neumann problem, characterized by problem~\eqref{eqs: Darcy} with $n = 3$ and $\partial_q \Omega = \partial \Omega$. In this case, the co-chain complex of interest is known as the de Rham complex with boundary conditions, given by
\begin{equation} \label{eq: de Rham 3D}
\begin{tikzcd}
H_0(\nabla) \arrow[r, "\nabla"] &
H_0(\curl)  \arrow[r, "\curl"] &
H_0(\div)   \arrow[r, "\div"] &
L^2 / \mathbb{R}.
\end{tikzcd}
\end{equation}

We note two important properties of this complex. First, we have the elementary identities $(\nabla) \circ (\curl) = 0$ and $(\curl) \circ (\div) = 0$. Second, the Helmholtz decomposition ensures that if $q \in H_0(\div)$ with $\div q = 0$, then a $r \in H_0(\curl)$ exists such that $q = \curl r$. Moreover, if $r \in H_0(\curl)$ and $\curl r = 0$, then $r = \nabla s$ for some $s \in H_0(\nabla)$.

Our solution technique exploits these properties of the complex. Let us proceed according to the following three steps.
\begin{enumerate}
    \item Given $f \in L^2 / \mathbb{R}$, let $q_f \in H_0(\div)$ be any function that satisfies
    \begin{align} \label{eq: Darcy step 1}
        \div q_f &= f.
    \end{align}

    \item Let $q_0 := q - q_f$. Since $\div q_0 = 0$, the Helmholtz decomposition ensures that a $r \in H_0(\curl)$ exists such that $\curl r = q_0$. This variable $r$ then has the property
    \begin{align*}
        K^{-1} \curl r = K^{-1} (q - q_f) =  -\nabla p - K^{-1} q_f.
    \end{align*}
    Note that $r$ cannot be found directly in this way since the equation is posed in $H_0(\div)$ for an unknown in $H_0(\curl)$. With the aim of obtaining a well-posed problem, we test this equation with functions $\curl \tilde r \in H_0(\div)$. We derive:
    \begin{align*}
        \langle K^{-1} \curl r, \curl \tilde r \rangle
        &= -\langle \nabla p, \curl \tilde r \rangle
        - \langle K^{-1} q_f, \curl \tilde r \rangle \\
        &= -\langle p, \div \curl \tilde r \rangle
        - \langle K^{-1} q_f, \curl \tilde r \rangle \\
        &= - \langle K^{-1} q_f, \curl \tilde r \rangle.
    \end{align*}
    Here, the second equality is due to integration by parts and the third follows from $\div \curl \tilde r = 0$.
    However, this equation still does not guarantee a unique solution $r$ because the curl operator has a non-zero kernel, which is given by the range of the gradient. We ensure orthogonality to this kernel by imposing $0 = \langle r, \nabla s \rangle = -\langle \div r, s \rangle$ for all $s \in H(\nabla)$. Thus, we introduce a term that penalizes $\div r$, giving us the problem:
    Find $r \in H_0(\curl) \cap H(\div)$ such that
    \begin{align} \label{eq: Darcy step 2}
        \langle K^{-1} \curl r, \curl \tilde r \rangle
        + \langle \div r, \div \tilde r \rangle
        &= - \langle K^{-1} q_f, \curl \tilde r \rangle,
    \end{align}
    for all $\tilde r \in H_0(\curl) \cap H(\div)$.

    \item We set $q := q_f + \curl r$ and it remains to compute the pressure variable: Find $p \in L^2 / \mathbb{R}$ such that
    \begin{align} \label{eq: Darcy step 3}
        \langle p, \div \tilde q \rangle &= \langle K^{-1} q, \tilde q \rangle,
        & \forall \tilde q \in H_0(\div).
    \end{align}
\end{enumerate}

The solvability of the systems \eqref{eq: Darcy step 1}--\eqref{eq: Darcy step 3} is discussed in the more general setting of the next subsection.

\begin{remark}
    Inhomogeneous boundary conditions can readily be incorporated in this procedure. In particular, the natural boundary condition $p = g_p$ on $\partial_p \Omega$ amounts to subtracting the term $\langle g_p, \nu \cdot (\curl \tilde r) \rangle_{\partial_p \Omega}$ from the right-hand side of \eqref{eq: Darcy step 2} and adding $\langle g_p, \nu \cdot \tilde q \rangle_{\partial_p \Omega}$ to the right-hand side of \eqref{eq: Darcy step 3}.

    On the other hand, the essential boundary condition $\nu \cdot q = g_q$ on $\partial_q \Omega$ requires first choosing a function $q_g \in H(\div)$ with $\nu \cdot q_g = g_q$. Then, we subtract $\div q_g$ from the right-hand side of \eqref{eq: Darcy step 1}, subtract $\langle K^{-1} q_g, \curl \tilde r \rangle$ from \eqref{eq: Darcy step 2}, and change the computation in Step 3 to $q := q_f + \curl r + q_g$.
    To avoid unnecessary distraction, we limit our exposition herein to the case of homogeneous boundary conditions.
\end{remark}

\subsection{The general case}
\label{sub: general}

In order to generalize the three-step procedure, we borrow notation from the setting of exterior calculus. In particular, each function space used in the previous section can be represented by $H \Lambda^k \subset L^2 \Lambda^k$ with $k \in [n - 3, n]$ and connected by differentials $\d_k : H\Lambda^k \to H\Lambda^{k + 1}$. We omit the subscript on $\d$ when no ambiguity arises and consider the co-chain complex $(H\Lambda^\bullet, \d)$, given by
\begin{equation} \label{eq: de Rham forms}
\begin{tikzcd}
H\Lambda^{n - 3} \arrow[r, "\d"] &
H\Lambda^{n - 2} \arrow[r, "\d"] &
H\Lambda^{n - 1} \arrow[r, "\d"] &
H\Lambda^n.
\end{tikzcd}
\end{equation}

Let us recall the defining property of an \emph{exact} complex, namely that
\begin{align}
    \Ker(\d_k) = \Ran(\d_{k - 1}).
\end{align}
This has two important implications. First, $\Ker(\d_k) \supseteq \Ran(\d_{k - 1})$ means that $\d_k \d_{k-1} = 0$. On the other hand, $\Ker(\d_k) \subseteq \Ran(\d_{k - 1})$ implies that if $p \in H\Lambda^k$ satisfies $\d p = 0$, then a $q \in H\Lambda^{k -1}$ exists with $\d_{k - 1} q = p$ and $q \perp \Ker(\d_{k - 1})$.
Moreover, we define both $\d_{n - 4}$ and $\d_n$ to be zero. In turn, we have $H \Lambda^{n - 3} \perp \Ker(\d)$ and $H \Lambda^n = \Ran(\d)$.

Let $\d^*$ be the adjoint of $\d$, i.e. $\langle \d^*p, q \rangle := \langle p, \d q \rangle$ for all $q \in H\Lambda^{k - 1}$ and sufficiently regular $p \in L^2 \Lambda^k$. To make precise the required regularity, we define the Sobolev spaces
\begin{align*}
    H \Lambda^k &:= \{ p \in L^2 \Lambda^k : \d p \in L^2 \Lambda^{k + 1} \}, \\
    H^* \Lambda^k &:= \{ p \in L^2 \Lambda^k : \d^*p \in L^2 \Lambda^{k - 1} \}.
\end{align*}

\begin{remark}
    Formally, the Sobolev space $H \Lambda^k$ defined here is a \emph{representation} of the vector space containing alternating, multi-linear $k$-forms on $\Omega$ and $\d$ is a representation of the exterior derivative. Here, we do not make this distinction and will work directly with the canonical representations of both the forms and the differentials. We refer the interested reader to \cite{arnold2018finite,spivak2018calculus}.
\end{remark}

The Darcy flow system \eqref{eqs: Darcy} can now be identified as a problem of the form: Find $(q,p) \in H\Lambda^{n - 1} \times H^*\Lambda^n$ that satisfies
\begin{subequations} \label{eqs: Darcy abstract}
\begin{align}
    K^{-1} q - \d^*p &= 0, \label{eq: Darcy law abstract} \\
    \d q &= f. \label{eq: Mass law abstract}
\end{align}
\end{subequations}
This formulation covers the three types of boundary conditions in 2D and 3D presented in Section~\ref{sub: model problem}. The corresponding spaces are presented in Table~\ref{tab: table 1} and their precise definitions can be found in Section~\ref{sub: preliminaries and notation}.

\begin{table}[ht]
    \caption{Explicit definitions of the spaces $H \Lambda^k$ for the different boundary conditions.}
    \label{tab: table 1}
    \centering

    \begin{tabular}{rllclclcl}
    \hline
    & & $H\Lambda^{n - 3}$ & & $H\Lambda^{n - 2}$ & & $H\Lambda^{n - 1}$ & & $H\Lambda^n$ \\
    \hline
    $n = 3$ & $\d$ & & $\nabla$ &  & $\curl$ &  & $\div$ & \\
    & Dirichlet & $H(\nabla) / \mathbb{R}$ & & $H(\curl)$ & & $H(\div)$ & & $L^2$ \\
    & Mixed & $H_{\partial_q \Omega}(\nabla)$ & & $H_{\partial_q \Omega}(\curl)$ & & $H_{\partial_q \Omega}(\div)$ & & $L^2$ \\
    & Neumann & $H_0(\nabla)$ & & $H_0(\curl)$ & & $H_0(\div)$ & & $L^2 / \mathbb{R}$ \\
    \hline
    $n = 2$ & $\d$ & & $\subset$ & &  $\rot$ & &  $\div$ & \\
    & Dirichlet & $0$ & & $H(\rot) / \mathbb{R}$ & & $H(\div)$ & & $L^2$ \\
    & Mixed & $0$ & & $H_{\partial_q \Omega}(\rot)$ & & $H_{\partial_q \Omega}(\div)$ & & $L^2$ \\
    & Neumann & $0$ & & $H_0(\rot)$ & & $H_0(\div)$ & & $L^2 / \mathbb{R}$ \\
    \hline
    \end{tabular}
\end{table}

The three steps of the solution procedure from Section~\ref{sub: Neumann 3D} can now be recast in terms of the spaces $H \Lambda^k$ and their associated differentials $\d$. For each step, we briefly show the solvability of the involved problem using standard arguments.
\begin{enumerate}
    \item \label{step: 1}
    Find $q_f \in H \Lambda^{n - 1}$ that satisfies
    \begin{align} \label{eq: general step 1}
        \d q_f &= f.
    \end{align}

    \begin{lemma}
        Problem \eqref{eq: general step 1} admits a solution.
    \end{lemma}
    \begin{proof}
        Existence is guaranteed by the fact that $f \in H \Lambda^n = \Ran(\d)$. We emphasize that $q_f$ is generally not unique.
        \qed
    \end{proof}
    \item \label{step: 2}
    Solve for $r \in H \Lambda^{n-2} \cap H^* \Lambda^{n-2}$:
    \begin{align} \label{eq: general step 2}
        (\d^*K^{-1}\d + \d\d^*)r = - \d^* K^{-1} q_f.
    \end{align}

    \begin{lemma}
        Problem \eqref{eq: general step 2} admits a unique solution
    \end{lemma}
    \begin{proof}
        Let $q_0 := q - q_f$, for which we have $\d q_0 = 0$. The exactness of the complex ensures that a $r \in H \Lambda^{n - 2}$ exists with $\d r = q_0$ and $r \perp \Ker(\d)$. This implies that $\d^* r = 0$ and so $r \in H^*\Lambda^{n - 2}$. Inserting $r$ in \eqref{eq: general step 2}, we see that
        \begin{align*}
            (\d^*K^{-1}\d + \d\d^*)r
            = \d^*K^{-1}\d r
            &= \d^*K^{-1}(q - q_f) \\
            &= \d^*\d^* p - \d^* K^{-1} q_f
            = - \d^* K^{-1} q_f.
        \end{align*}
        In turn, existence is verified and it remains to show uniqueness. Considering a zero right-hand side, we test the equation with $r \in H \Lambda^{n - 2} \cap H^* \Lambda^{n-2}$ and derive
        \begin{align*}
            \langle K^{-1} \d r, \d r \rangle
            + \langle \d^* r, \d^* r \rangle
            &= 0.
        \end{align*}
        Since $K^{-1}$ is positive definite, it follows that $r \in \Ker(\d_{n - 2}) \cap \Ker(\d_{n - 3}^*)$. Now, since
        \begin{align*}
            \Ker(\d_{n - 3}^*) = \Ran(\d_{n - 3})^{\perp} = \Ker(\d_{n - 2})^{\perp},
        \end{align*}
        we have $r = 0$ and uniqueness is shown.
    \qed
    \end{proof}
    \item \label{step: 3}
    Construct $q := q_f + \d r$. Solve for $p \in H^* \Lambda^n$:
    \begin{align} \label{eq: general step 3}
        \d^* p
        &= K^{-1} q.
    \end{align}
    \begin{lemma}
        Problem \eqref{eq: general step 3} admits a unique solution.
    \end{lemma}
    \begin{proof}
        In this case, existence is verified by the true solution $p$ to the original problem \eqref{eqs: Darcy abstract}. For uniqueness, we note that a zero right-hand side is equivalent to stating that $p \perp \Ran(\d)$. However, since $H \Lambda^n = \Ran(\d)$, we conclude that $p = 0$.
    \qed
    \end{proof}
\end{enumerate}

To conclude this section, we briefly show that our three-step procedure constructs the unique solution to the original problem.

\begin{lemma}
    The pair $(q, p)$ obtained from the three-step procedure solves \eqref{eqs: Darcy abstract}.
\end{lemma}
\begin{proof}
    First, \eqref{eq: Darcy law abstract} is satisfied due to \eqref{eq: general step 3}. Second, \eqref{eq: Mass law abstract} is fulfilled by the calculation $\d q = \d(q_f + \d r) = \d q_f = f$.
\end{proof}


\subsection{Discretization of Darcy flow using structure-preserving methods}
\label{sub: Discretization}

We continue with the discrete setting in which we let $P\Lambda^k \subset H\Lambda^k$ be a finite-dimensional subspace for each $k$. The differential $\d_{h, k} : P\Lambda^k \to P\Lambda^{k + 1}$ is defined as the restriction of $\d_k$ to $P\Lambda^k$. We often omit the subscript $k$ on the differential and we assume that $(P\Lambda^{\bullet}, \d_h)$ forms an exact complex. Such exact discrete complexes form an active area of research, see e.g. \cite{arnold2018finite,arnold2006finite,hu2022family}.

As our main example, we focus on the family of trimmed elements of polynomial order $r$, i.e. $P\Lambda^k := P_r^- \Lambda^k$ in the notation of Finite Element Exterior Calculus \cite{arnold2006finite}. This family consists of the Lagrange elements $\mathbb{L}_r$, the N\'ed\'elec\cite{nedelec1980mixed} element of the first kind $\mathbb{N}_{r - 1}$, the Raviart-Thomas\cite{raviart1977mixed} element $\mathbb{RT}_{r - 1}$, and the discontinuous, piecewise polynomials $\mathbb{P}_{r - 1}$. The elements of lowest order, with $r = 1$, are referred to as the Whitney forms, and form the exact complex $(P\Lambda^{\bullet}, \d_h)$:

\begin{equation} \label{eq: de Rham discrete}
\begin{tikzcd}
\mathbb{L}_1 \arrow[r, "\nabla"] &
\mathbb{N}_0 \arrow[r, "\curl"] &
\mathbb{RT}_0 \arrow[r, "\div"] &
\mathbb{P}_0.
\end{tikzcd}
\end{equation}



In the case of homogeneous boundary conditions on the variable $q \in P \Lambda^{n - 1}$, i.e. $\partial_q \Omega \neq \emptyset$, we consider the subspaces $P_{\partial_q \Omega} \Lambda^k$ in which the degrees of freedom on $\partial_q \Omega$ are set to zero.

The discrete complex $(P\Lambda^\bullet, \d_h)$ is exact (see e.g. \cite{arnold2006finite}) and therefore the three-step technique proposed in Section~\ref{sub: general} is directly applicable.

The first step \eqref{eq: general step 1} can be solved using \emph{any} locally conservative scheme, e.g. with a finite volume method with a two-point flux approximation (TPFA). This leads to a small system consisting only of cell-center pressure unknowns and is therefore relatively inexpensive to compute. Since the TPFA method is not consistent in general, the remaining two steps can be seen as corrections.

The second step requires the operators $\d_h^*\d_h$ and $\d_h\d_h^*$. The former can directly be implemented as
$\langle \d_h r, \d_h \tilde r \rangle = \langle \d r, \d \tilde r \rangle$ since $P\Lambda^k \subset H \Lambda^k$. For the latter operator $\d_h \d_h^*$, we first solve for $\d_h^*r \in P\Lambda^{k - 1}$:
\begin{align} \label{eq: def discrete adjoint}
    \langle \d_h^* r,  \tilde s \rangle
    &= \langle r, \d_h \tilde s \rangle,
    & \forall  \tilde s \in P\Lambda^{k - 1},
\end{align}
and then compute $\langle \d_h^* r, \d_h^* \tilde r \rangle$.
This is computationally costly, so we propose a mass lumping technique on the mass matrix of $P\Lambda^{k - 1}$ in Section~\ref{sec: implementation}. This modification does not change the solution since the only purpose of the term $\d_h\d_h^*$ is to penalize $\d_h^* r$.

\begin{remark}
    We emphasize that \eqref{eq: def discrete adjoint} is solvable for all $r \in P\Lambda^k$ and so $\d_h^*: P\Lambda^k \to P\Lambda^{k - 1}$ is a well-defined operator. However, we generally have $P\Lambda^k \not \subseteq H^* \Lambda^k$, so $\d_h^*$ is not a restriction of $\d^*$. For example, the piecewise constants $\mathbb{P}_0 = P\Lambda^n \subset H \Lambda^n = L^2$ are not contained in $H^* \Lambda^n = H(\nabla)$ but \eqref{eq: def discrete adjoint} nevertheless defines $\d_{h, n - 1}^*$ as a discrete gradient on this space. Analogously, the $\d_{h, k}^*$ operators correspond to a discrete curl on $\mathbb{RT}_r$ and a discrete divergence on $\mathbb{N}_r$.
\end{remark}

The solutions to systems \eqref{eq: general step 3} and \eqref{eq: general step 1} can be obtained by solving an elliptic problem posed on $P\Lambda^n$ and thus, in the lowest order case, concerns only cell-center variables. Similarly, the problem in step \ref{step: 2} concerns degrees of freedom on the mesh edges in 3D and nodes in 2D. We have thus partitioned the original saddle point formulation into three smaller, elliptic problems posed on either the nodes, edges, or cells of the mesh.

Finally, we note that the discretization can be generalized to polyhedral meshes using the virtual element method \cite{beirao2013basic}. In the lowest order case, this amounts to defining a degree of freedom of $P\Lambda^k$ on each $k$-dimensional mesh entity, i.e. on the nodes, edges, faces, or cells.

\subsection{Flow in fractured porous media}
\label{sub: MiDi}

Next, we consider a model of flow in fractured porous media in which the fractures are represented by lower-dimensional manifolds. We start by presenting the subdivision of the domain into subdomains of different dimensions, then define the finite element spaces and governing equations, and finally introduce the relevant exact discrete complex. Thus, we limit ourselves to the discrete setting and refer the interested reader to \cite{boon2021functional,boon2018robust} for the continuous case.

The first step is to partition the domain of computation into the $n$-dimensional bulk matrix, the $(n - 1)$-dimensional fractures and the lower-dimensional intersection lines and points. Specifically, let $\Omega$ be partitioned into open subdomains $\Omega_i$ with $i \in I$ the index and $d_i$ its dimensionality. We assume that each $\Omega_i$ with $d_i < n$ has at least one neighbor $\Omega_j$ such that $d_j = d_i + 1$ and $\Omega_i$ coincides with a part of $\partial \Omega_j$, denoted by $\partial_i \Omega_j$. A precise definition of allowable geometries is given in \cite{boon2021functional}.

On each subdomain $\Omega_i$, we introduce a shape-regular, simplicial mesh $\Omega_{h, i}$. We impose that the meshes are matching in the sense that each $d_i$-dimensional cell of $\Omega_{h, i}$ coincides with a face on $\partial_i \Omega_{h, j}$ for all its neighbors with $d_j = d_i + 1$. Since we only consider the discrete case in this work, we abuse notation and omit the subscript $h$ on $\Omega_{h, i}$.

We group the finite element spaces from Section~\ref{sub: Discretization} as in \cite{nordbotten2017modeling} to define the mixed-dimensional spaces:
\begin{align}
    P\mathfrak{L}^k := \prod_{\substack{ i \in I \\ d_i \ge n - k}} P\Lambda^{d_i - (n - k)}(\Omega_i)
\end{align}

For ease of reference, Table~\ref{tab: table 3} presents the local finite element spaces depending on the dimensionality $d_i$ of the subdomain.

\begin{table}[ht]
    \caption{A mixed-dimensional family of finite elements of lowest order that form an exact discrete complex for $n = 3$.}
    \label{tab: table 3}
    \centering

    \begin{tabular}{lllll}
    \hline
    $d_i$ & $P \mathfrak{L}^0$ & $P \mathfrak{L}^1$ & $P \mathfrak{L}^2$ & $P \mathfrak{L}^3$ \\
    \hline
    $3$ & $\mathbb{L}_1$ & $\mathbb{N}_0$ & $\mathbb{RT}_0$ & $\mathbb{P}_0$ \\
    $2$ & & $\mathbb{L}_1$ & $\mathbb{RT}_0$ & $\mathbb{P}_0$ \\
    $1$ & & & $\mathbb{L}_1$ & $\mathbb{P}_0$ \\
    $0$ & & & & $\mathbb{P}_0$ \\
    \hline
    \end{tabular}
\end{table}

We are interested in a mixed formulation and therefore introduce the flux $\mathfrak{q} \in P \mathfrak{L}^{n - 1}$ and pressure $\mathfrak{p} \in P \mathfrak{L}^n$. Since these are variables defined on subdomains of different dimensionalities, we refer to them as mixed-dimensional and denote them using a Gothic font. We revert to standard font to indicate a restriction to a subdomain, i.e. $p_i := \mathfrak{p}|_{\Omega_i}$.

With the function spaces defined, we continue with the governing equations of mixed-dimensional Darcy flow \cite{berre2021verification}, which are a generalization of \cite{martin2005modeling}. First, we assume Darcy's law tangential to each $\Omega_i$ and normal to each $\partial_i \Omega_j$, with conductivities $K_i$ and $K_{ij}$, respectively. Second, the mass balance equation relates the tangential flux to the contribution from higher-dimensional neighboring subdomains. This leads us to the following equations:
\begin{subequations} \label{eqs: MD Darcy tedious}
\begin{align}
    q_i + K_{i} \nabla_i p_i &= 0
    & \text{on }&\Omega_i,
    & 1 &\le d_i \le n, \\
    \nu_j \cdot q_j + K_{ij} (p_i - p_j) &= 0
    & \text{on }&\partial_i \Omega_j,
    & 0 &\le d_i \le n - 1, \\
    \nabla_i \cdot q_i + \sum_{\substack{j \in I \\ d_j = d_i + 1}} (- \nu_j \cdot q_j)|_{\partial_i \Omega_j} &= f_i
    & \text{on }&\Omega_i,
    & 0 &\le d_i \le n. \label{eq: MD mass balance}
\end{align}
Here, $\nabla_i$ is the del-operator on $\Omega_i$ and $\nu_j$ is the outward oriented, unit normal vector to $\partial \Omega_j$. In \eqref{eq: MD mass balance}, we assume that the first term is zero for $d_i = 0$ and the second term is zero for $d_i = n$.

The following boundary conditions are imposed:
\begin{align} \label{eq: MD Darcy bcs}
    p_i &= 0 & \text{on } &\partial \Omega_i \cap \partial \Omega, &
    \nu_i \cdot q_i &= 0 & \text{on } &\partial_0 \Omega_i.
\end{align}
with $\partial_0 \Omega_i \subseteq \partial \Omega_i$ the fracture tips, i.e. portion of the boundary of $\Omega_i$ that does not border a lower-dimensional subdomain.
\end{subequations}

In order to show that \eqref{eqs: MD Darcy tedious} has the structure \eqref{eqs: Darcy abstract}, we first introduce the following inner products:
\begin{align*}
    \langle \mathfrak{p}, \mathfrak{\tilde p} \rangle_{L^2 \mathfrak{L}^n}
    &:= \sum_{i \in I}
    \langle p_i, \tilde p_i \rangle_{\Omega_i},
    \\
    \langle \mathfrak{q}, \mathfrak{\tilde q} \rangle_{L^2 \mathfrak{L}^{n - 1}}
    &:= \sum_{\substack{i \in I \\ d_i \ge 1}}
    \bigg(
    \langle q_i, \tilde q_i \rangle_{\Omega_i}
    +
    \sum_{\substack{j \in I \\ d_j = d_i - 1}}
    \langle \nu_i \cdot q_i, \nu_i \cdot \tilde q_i \rangle_{\partial_j \Omega_i}
    \bigg).
\end{align*}

Second, we define mixed-dimensional divergence $(\mathfrak{D}\cdot): P \mathfrak{L}^{n - 1} \to P \mathfrak{L}^n$ as
\begin{align*}
    (\mathfrak{D} \cdot \mathfrak{q})|_{\Omega_i} &:= \nabla_i \cdot q_i + \sum_{\substack{j \in I \\ d_j = d_i + 1}} (- \nu_j \cdot q_j)|_{\partial_i \Omega_j}
    & \forall i &\in I.
\end{align*}

Finally, we collect the source terms to create $\mathfrak{f} \in P\mathfrak{L}^n$ such that $\mathfrak{f}|_{\Omega_i} = f_i$ and similarly, we define $\mathfrak{K}$ such that it equals $K_i$ on $\Omega_i$ and $K_{ij}$ on $\partial_i \Omega_j$.
The weak formulation of the fracture flow problem \eqref{eqs: MD Darcy tedious} then becomes (cf. \cite{boon2018robust} for the derivation): Find $(\mathfrak{q}, \mathfrak{p}) \in P \mathfrak{L}^{n - 1} \times P \mathfrak{L}^n$ such that
\begin{subequations} \label{eqs: MD Darcy system}
\begin{align}
    \langle \mathfrak{K}^{-1} \mathfrak{q}, \mathfrak{\tilde q} \rangle_{L^2 \mathfrak{L}^{n - 1}}
    - \langle \mathfrak{p}, \mathfrak{D} \cdot \mathfrak{\tilde q} \rangle_{L^2 \mathfrak{L}^n}
    &= 0,
    & \forall \mathfrak{\tilde q} &\in P \mathfrak{L}^{n - 1}, \\
    \langle \mathfrak{D} \cdot \mathfrak{q}, \mathfrak{\tilde p} \rangle_{L^2 \mathfrak{L}^n}
    &= \langle \mathfrak{f}, \mathfrak{\tilde p} \rangle_{L^2 \mathfrak{L}^n},
    & \forall \mathfrak{\tilde p} &\in P \mathfrak{L}^n.
\end{align}
\end{subequations}

We observe that \eqref{eqs: MD Darcy system} has the structure \eqref{eqs: Darcy abstract}. Hence, we next require a mixed-dimensional curl operator $(\mathfrak{D} \times)$ in order to generate a solenoidal field. It is important to note that here, solenoidal means that $\mathfrak{D} \cdot \mathfrak{q} = 0$ and this is not the same as imposing $\nabla_i \cdot q_i = 0$ on all $i \in I$.
In order to apply our proposed solution technique, we therefore require the mixed-dimensional analogues of the curl and gradient. These differential operators were introduced in \cite{boon2021functional} and form the following co-chain complex:
\begin{equation} \label{eq: MD-de Rham}
\begin{tikzcd}
P\mathfrak{L}^0 \arrow[r, "\mathfrak{D}"] &
P\mathfrak{L}^1 \arrow[r, "\mathfrak{D} \times"] &
P\mathfrak{L}^2 \arrow[r, "\mathfrak{D} \cdot"] &
P\mathfrak{L}^3.
\end{tikzcd}
\end{equation}

We recall that the spaces $P\mathfrak{L}^k$ are given by the columns of Table~\ref{tab: table 3}. Moreover, the mixed-dimensional gradient $(\mathfrak{D})$ and curl $(\mathfrak{D} \times)$ are defined as follows:
\begin{subequations}
\begin{align}
    (\mathfrak{D} \mathfrak{s})|_{\Omega_i} &:=
    \begin{cases}
        \nabla_i s_i, & d_i = 3, \\
        \sum_{\substack{j \in I \\ d_j = 3}} (- (\nu_i \cdot\nu_j) s_j)|_{\partial_i \Omega_j}, & d_i = 2.
    \end{cases} \\
    (\mathfrak{D} \times \mathfrak{r})|_{\Omega_i} &:=
    \begin{cases}
        \nabla_i \times r_i, & d_i = 3, \\
        \nabla_i^\perp r_i + \sum_{\substack{j \in I \\ d_j = 3}} (\nu_j \times r_j)|_{\partial_i \Omega_j}, & d_i = 2, \\
        \sum_{\substack{j \in I \\ d_j = 2}} ((\nu_j^\perp \cdot \tau_i) r_j)|_{\partial_i \Omega_j}, & d_i = 1.
    \end{cases}
\end{align}
\end{subequations}
Here, $\nu_i|_{\partial_i \Omega_j}$ for $d_i = 2$ is the unit vector normal to $\Omega_i$ that forms a positive orientation with the tangent bundle of $\Omega_i$, according to the right-hand rule. Moreover, $\tau_i$ for $d_i = 1$ is the unit vector tangent to $\Omega_i$.

The fact that \eqref{eq: MD-de Rham} is exact was shown in \cite{licht2017complexes,boon2021functional} and, in turn, the solution technique of Section~\ref{sub: general} is directly applicable.

\section{A Reduced Basis Method ensuring local mass conservation}
\label{sec: RBM}

The aim of this section is to augment the solution technique proposed in Section~\ref{sec: solution technique} by replacing step \ref{step: 2} with a reduced basis method. The mapping that we aim to approximate is $(K, q_f) \to r$. We utilize a splitting into a computationally costly \emph{off-line} stage and an efficient \emph{on-line} stage. In the \emph{off-line} stage, we first compute the mapping $f \to q_f$ given by \eqref{eq: general step 1}, e.g. by saving an $LU$-decomposition. Then, we construct a reduced basis approximation to the mapping $(K, q_f) \to r$ given by system \eqref{eq: general step 2}. The details of this construction are given in Section~\ref{sub: RBM construction}. 

The \emph{on-line} stage then amounts to the following steps:
\begin{enumerate}
    \item Given $f$, construct $q_f$ by solving \eqref{eq: general step 1}.
    \item Given $K$ and $q_f$, compute $r$ using the reduced basis.
    \item Compute $q := q_f + \d r$ and construct $p$ by solving \eqref{eq: general step 3}.
\end{enumerate}

We emphasize that the solution obtained from this method is guaranteed to conserve mass locally, hence achieving our main goal. In fact, the error arising from the reduced basis approximation is contained in the contribution $\d r$, which is divergence-free by construction.

\subsection{Construction of the reduced basis by Proper Orthogonal Decomposition}
\label{sub: RBM construction}

Let us focus on the second step in this algorithm in the discrete setting. Then Problem~\eqref{eq: general step 2} is of the form:
\begin{align} \label{eq: step 2 discrete}
    \mathsf{A^K} \mathsf{r} = \mathsf{b_f}
\end{align}
In which the matrix $\mathsf{A^K}$ depends on the material parameter $K$ and the vector $\mathsf{b_f}$ on the right-hand side is determined by the source term $f$. In the ``offline'' stage of the method, we now construct a reduced basis that captures the influence of these parameters on the solution. We use the conventional Proper Orthogonal Decomposition approach to achieve this.

In particular, we first choose $n_S$ values for the parameter pair $(K, f)$ and solve \eqref{eq: step 2 discrete} for each value pair. For $1 \le i \le n_S$, the solution vectors $\mathsf{r}_i \in \mathbb{R}^{n_r}$ are known as \emph{snapshots}, and we collect these to form the columns of the matrix $\mathsf{S} \in \mathbb{R}^{n_r \times n_S}$.

Next, we compute the singular value decomposition of $\mathsf{S}$ such that
\begin{align} \label{eq: SVD}
    \mathsf{S} = \mathsf{U \Sigma V^T}
\end{align}
In which $\mathsf{U} \in \mathbb{R}^{n_r \times n_r}$ and $\mathsf{V} \in \mathbb{R}^{n_S \times n_S}$ are orthogonal matrices and $\mathsf{\Sigma}$ is a diagonal matrix containing the singular values $\sigma_i$. For given threshold value $\varepsilon$, we select $n_m$ as the smallest index that satisfies
\begin{align}
    \sigma_{n_m}
    \ge \varepsilon
\end{align}

We then extract the $n_m$ most important modes by restricting $\mathsf{U}$ to its first $n_m$ columns, creating $\mathsf{U_m} \in \mathbb{R}^{n_r \times n_m}$. The reduced problem now becomes:
\begin{align*}
    \mathsf{(U_m^T A^K U_m) r_m} = \mathsf{U_m^T b_f}.
\end{align*}

Note that this is a system with $n_m$ unknowns and since we typically have $n_m \ll n_r$, it is significantly less expensive to solve than the original system \eqref{eq: step 2 discrete}. Finally, we have $\mathsf{r} = \mathsf{U_m r_m}$ as the reduced basis approximation to the full order system.

\section{Numerical results}
\label{sec: Numerics}

This section concerns the implementation, set-up, and results of the numerical experiments. Section~\ref{sec: implementation} provides guidance into the numerical implementation of the proposed procedure and Section~\ref{sec: Numerical experiments} presents the results.

\subsection{Implementation}
\label{sec: implementation}

Since the spaces are finite dimensional, we can represent each variable $p \in P\Lambda^k$ as a vector $\mathsf{p} \in \mathsf{P}\Lambda^k := \mathbb{R}^{n_k}$ containing the values of its $n_k$ degrees of freedom. Moreover, the linear operators can be represented by matrices, e.g. the mass matrix $\mathsf{M_k}$ is given by:
\begin{align*}
    \mathsf{\tilde p^T M_k p} &= \langle p , \tilde p \rangle.
\end{align*}
Similarly, let $\mathsf{B_k}$ be the matrix representation of the differential $\d_{h, k}$:
\begin{align*}
    \mathsf{p^T \hat B_k q} &:= \langle \d_{h, k} q, p \rangle, &
    \mathsf{B_k} &:= \mathsf{M_{k + 1}^{-1} \hat B_k}.
\end{align*}
\begin{remark}
    Depending on the implementation, it may be easier to compute $\mathsf{B_k}$ directly and set $\mathsf{\hat B_k = M_{k + 1} B_k}$. In fact, the degrees of freedom in $P\Lambda^k$ and $P\Lambda^{k + 1}$ can be chosen such that all non-zero entries of $\mathsf{B_k}$ are $\pm 1$.
\end{remark}

These matrices allow us to compute the operator $\d_{h, k}^* \d_{h, k}$ as
\begin{align}
    \langle \d_{h, k} q, \d_{h, k} \tilde q \rangle
    =
    \mathsf{
    \tilde q^T (\hat B_k^T M_{k + 1}^{-1} \hat B_k) q = \tilde q^T (B_k^T M_{k + 1} B_k) q
    }.
\end{align}

On the other hand, the penalization term $\d_{h, k} \d_{h, k}^*$ requires the more involved computation
\begin{align}
    \langle \d_{h, k}^* q,  \d_{h, k}^*q \rangle
    =
    \mathsf{
    \tilde q^T (\hat B_k M_k^{-1} \hat B_k^T) q
    }.
\end{align}

The inversion of the mass matrix $\mathsf{M_k}$ is typically not feasible and, in our case, not necessary. We proceed by letting $\mathsf{L_k}$ be an easily invertible matrix obtained after mass lumping of $\mathsf{M_k}$. This leads us to the following approximation:
\begin{align}
    \langle \d_{h, k}^* q,  \d_{h, k}^*q \rangle
    \approx
    \mathsf{
    \tilde q^T (\hat B_k L_k^{-1} \hat B_k^T) q
    }.
\end{align}
For $k = n$, we note that $\d_{h, k} \d_{h, k}^*$ corresponds to a discrete Laplace operator and we can choose $\mathsf{L_{n - 1}}$ such that the resulting scheme is a TPFA finite volume method \cite{baranger1996connection,boffi2013mixed} for the Laplace equation. This is also possible in the mixed-dimensional case \cite{boon2020convergence}.

Finally, the conductivity induces a scaled inner product and we denote the corresponding matrix by $\mathsf{M_{n - 1}^K}$, i.e.
\begin{align*}
    \mathsf{\tilde q^T M_{n - 1}^K q} &= \langle K^{-1} q , \tilde q \rangle.
\end{align*}

With these matrices and vectors defined, we now repeat the three-step procedure to guide implementation:
\begin{enumerate}
    \item Solve the following system for $\mathsf{p_f} \in \mathsf{P}\Lambda^n$:
    \begin{align}
        \mathsf{(\hat B_{n - 1} L_{n - 1}^{-1} \hat B_{n - 1}^T) p_f} = \mathsf{f},
    \end{align}
    and set $\mathsf{q_f} = \mathsf{L_{n - 1}^{-1} \hat B_{n - 1}^T p_f}$.
    \item Use the reduced basis method to approximate $\mathsf{r} \in \mathsf{P}\Lambda^{n - 2}$ that satisfies
    \begin{align}
        (\mathsf{B_{n - 2}^T M_{n - 1}^{K} B_{n - 2}
        + \hat B_{n - 3} L_{n - 3}^{-1} \hat B_{n - 3}^T} ) \mathsf{r}
        =
        (- \mathsf{B_{n - 2}^T M_{n - 1}^K}) \mathsf{q_f}.
    \end{align}
    \item Set $\mathsf{q} = \mathsf{q_f + B_{n - 2}} \mathsf{r}$. Solve the following system for $\mathsf{p} \in \mathsf{P}\Lambda^n$:
    \begin{align}
        \mathsf{
            (\hat B_{n - 1} L_{n - 1}^{-1} \hat B_{n - 1}^T) p = (\hat B_{n - 1} L_{n - 1}^{-1} M_{n - 1}^K) q
        }.
    \end{align}
\end{enumerate}

It is important to note that the first and third step amount to solving the same system. Thus, for computational efficiency, we save the $LU$-decomposition of this matrix in the \emph{off-line} stage.

Moreover, the replacement of $\mathsf{M_{n - 3}^{-1}}$ by $\mathsf{L_{n - 3}^{-1}}$ in the second step changes the solution $\mathsf{r}$ but not $\mathsf{B_{n - 2} r}$. Formally, the mass-lumping changes the orthogonality condition with respect to $\Ker(\mathsf{B_{n - 2}})$ by employing a different inner product. However, the augmentation lies in the kernel of the curl operator and therefore does not affect the final solution $\mathsf{q}$.

\begin{remark}
    For simplicity of implementation, we may set $\mathsf{L_k = I}$ and substitute $\mathsf{B_{n - 3}}$ for $\mathsf{\hat B_{n - 3}}$ in the second step. This lets the first and third step become purely geometrical in the sense that the matrices only depend on the connectivity of the mesh entities. We note that this simplification will not affect the solvability of the systems, or the final solution, but requires proper scaling with the mesh size $h$.
\end{remark}

\begin{remark}
    If we set $\mathsf{L_{n - 1} = M_{n - 1}^K}$, then the first step amounts to solving the Schur complement system and the true solution would be obtained directly. From an algebraic perspective, the proposed technique therefore corresponds to an approximation of the Schur complement system with suitable corrections.
\end{remark}

Finally, we emphasize that this implementation is valid for both the mixed finite element method and the virtual element method \cite{beirao2013basic}. The difference lies in the computation of the mass matrices.

For this work, the numerical experiments were implemented using PyGeoN \cite{pygeon}, an open-source Python package. The mixed-dimensional structure was used from PorePy \cite{keilegavlen2021porepy}, with the grids created using the meshing software GMSH \cite{geuzaine2009gmsh}. Finally, the mixed-dimensional curl operator was adapted from \cite{budisa2020mixed}.

\subsection{Numerical experiments}
\label{sec: Numerical experiments}

We investigate the performance of our proposed technique using three test cases, of varying complexity. In order to exhibit the wide applicability, the first case simulates a three-dimensional layered porous medium problem, the second concerns a two-dimensional problem with fractures on a polygonal mesh, and the third consists of a three-dimensional fractured porous medium problem. In each case, we vary the model parameters and follow Section~\ref{sub: RBM construction} to generate a reduced basis. The resulting method is compared to a reference solution, obtained by solving the original, full-order model for a representative choice of parameters.

\begin{table}[ht]
    \caption{The solution procedure solves three smaller, symmetric positive definite systems instead of the full-order model (FOM), which has a saddle-point structure. By approximating the second step with a reduced basis method (RBM) using a threshold value of $\varepsilon = 10^{-7}$, a small error is introduced but the local mass balance is preserved.}
    \label{tab: summary numerics}
    \centering

    \begin{tabular}{r|rrr|rrr}
    \hline \rule{0pt}{2.6ex}
     & \multicolumn{3}{c|}{Number of degrees of freedom}
     & \multicolumn{3}{c}{Relative error due to RBM} \\
    Case & FOM & Steps 1,3 & Step 2 (RBM) &
    Pressure & Flux & Mass \\
    \hline \rule{0pt}{2.6ex}
        1 & 75,264 & 24,576 & 31,024 (19)
        & 1.59e-07 & 1.91e-07 & 3.94e-16\\
        2 & 2,868 & 993 & 865 (53)
        & 6.37e-05 & 1.14e-07 & 2.47e-17\\
        3 & 123,583 & 39,699 & 53,007 (27)
        & 5.59e-10 & 3.41e-08 & 5.07e-16\\
    \hline
    \end{tabular}
\end{table}

We summarize the main observations next, based on the results presented in Table~\ref{tab: summary numerics}.
First, we note that the number of degrees of freedom is reduced by a factor three, approximately, when comparing the full-order model to the systems in the first and third step. Recall that these are cell-centered, symmetric positive definite problems and are therefore amenable to a range of efficient solvers.

Second, we observe in all cases that the solution to the original problem is recovered if the system \eqref{eq: general step 2} in the second step is solved exactly. This verifies the exactness of the discrete complexes.

Third, our procedure provides a solid basis for the use of inexact solvers, as we explore here with RBM, because the introduced error can be contained to the constitutive law. This is reflected in the final three columns of Table~\ref{tab: summary numerics}, which shows that the mass balance equation is satisfied up to machine precision. The reported values concerning the pressure and flux are relative errors with respect to the $L^2$-norm. For the flux, this is equivalent to the relative error in the $H(\div)$ norm, due to the local mass conservation.

Fourth, the choice of a threshold value on the singular values is directly reflected in the accuracy of the solution. By lowering this threshold, more modes are used and a better approximation of the solution is obtained. This trade-off between accuracy and computational cost can be adjusted according to the problem at hand.

Fifth, we observe numerically that the number of significant modes does not dependent strongly on the mesh size. This can be explained by the fact that the response of the solution to the model parameter is the same on different meshes. In turn, a reduced basis formed on a coarse mesh can provide valuable insight for finer meshes, allowing for an efficient choice on the number of necessary snapshots.

We continue this section with separate descriptions of the three numerical test cases and present the corresponding, case-specific observations.
In each case, we set the following, parametrized boundary condition:
\begin{align}
    p(\bm{x})|_{\bm{x} \in \partial \Omega} &= \bar{\bm{\alpha}} \cdot \bm{x}, &
    \bar{\bm{\alpha}} &\in [0, 1]^n,
\end{align}
with $n$ the dimension of $\Omega$.

\subsubsection{A three-dimensional, layered porous medium}

As a first test case, we consider a set-up that emulates a layered porous medium. Let the domain $\Omega$ be the unit cube, subdivided into four equal, horizontal slabs. The bottom and third layer form the subdomain $\Omega_0$ whereas the remaining two layers form $\Omega_1$. We parametrize the conductivity, source term, and boundary conditions as follows:
\begin{align*}
    K(\bm{x}) &=
    \begin{cases}
        1, & \bm{x} \in \Omega_0,   \\
        \bar{K}, & \bm{x} \in \Omega_1,
    \end{cases} & &
    \begin{aligned}
        \\
        \bar{K} \in \left[10^{-5}, 10^{5}\right],
    \end{aligned} \\
    f(\bm{x}) &= \bar{f}, &
    &\bar{f} \in [-1, 1]
    .
\end{align*}

The problem is discretized using the Raviart-Thomas pair of lowest order $(\mathbb{RT}_0, \mathbb{P}_0)$ on a regular, tetrahedral grid with typical mesh size $h := 2^{-3}$. The parameter values for the snapshots are generated using a Latin hypercube sampling in which $\bar{f}$ and $\bar{\bm{\alpha}}$ are equi-distributed and $\bar{K}$ is distributed log-uniformly.
We generate 44 snapshots and compute the singular value decomposition \eqref{eq: SVD}. The singular values $\sigma_i$ are illustrated in Figure~\ref{fig: case 1}(left).

\begin{figure}[htb]
    \centering
    \includegraphics[height=4.4cm]{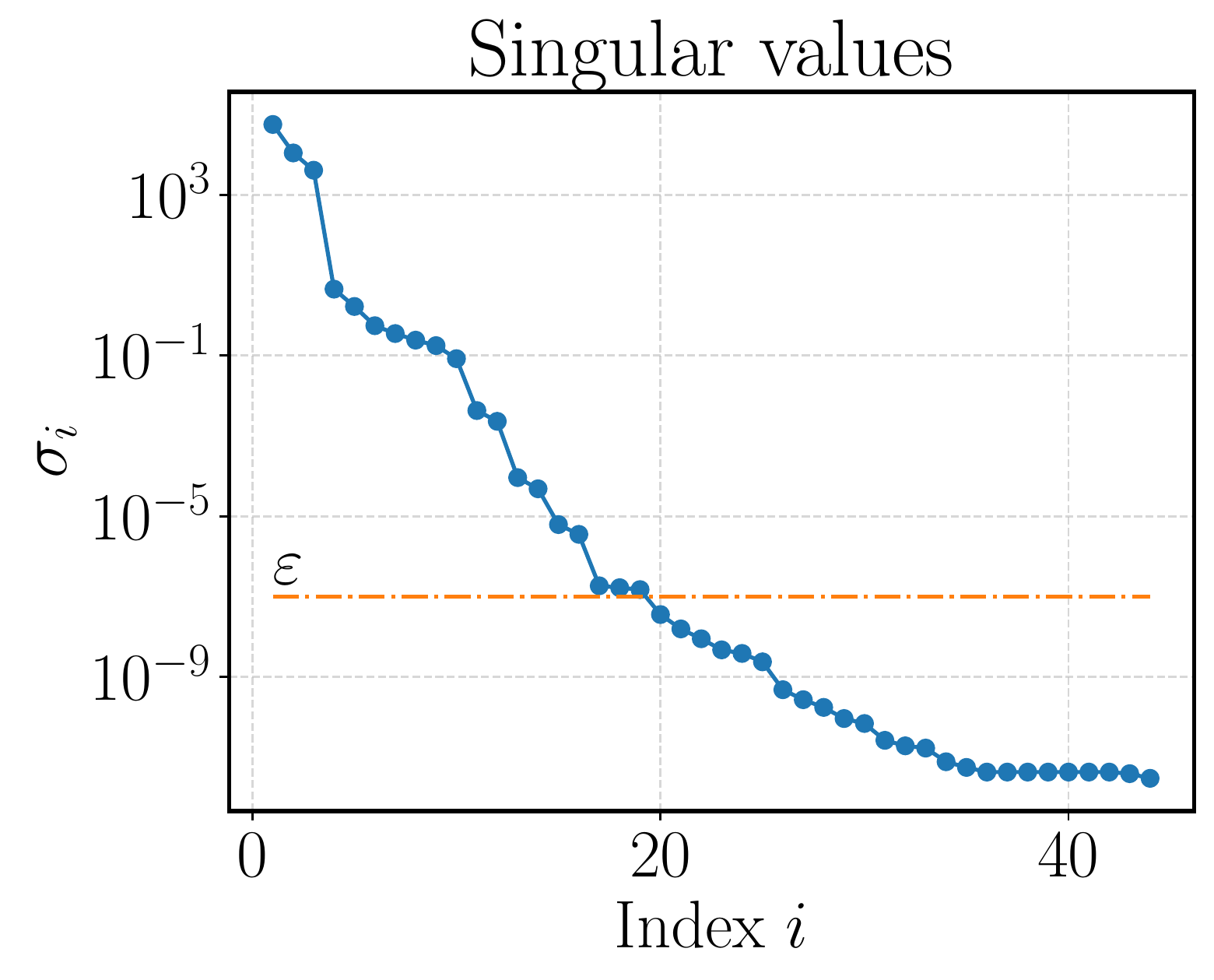}
    \
    \includegraphics[height=4.4cm]{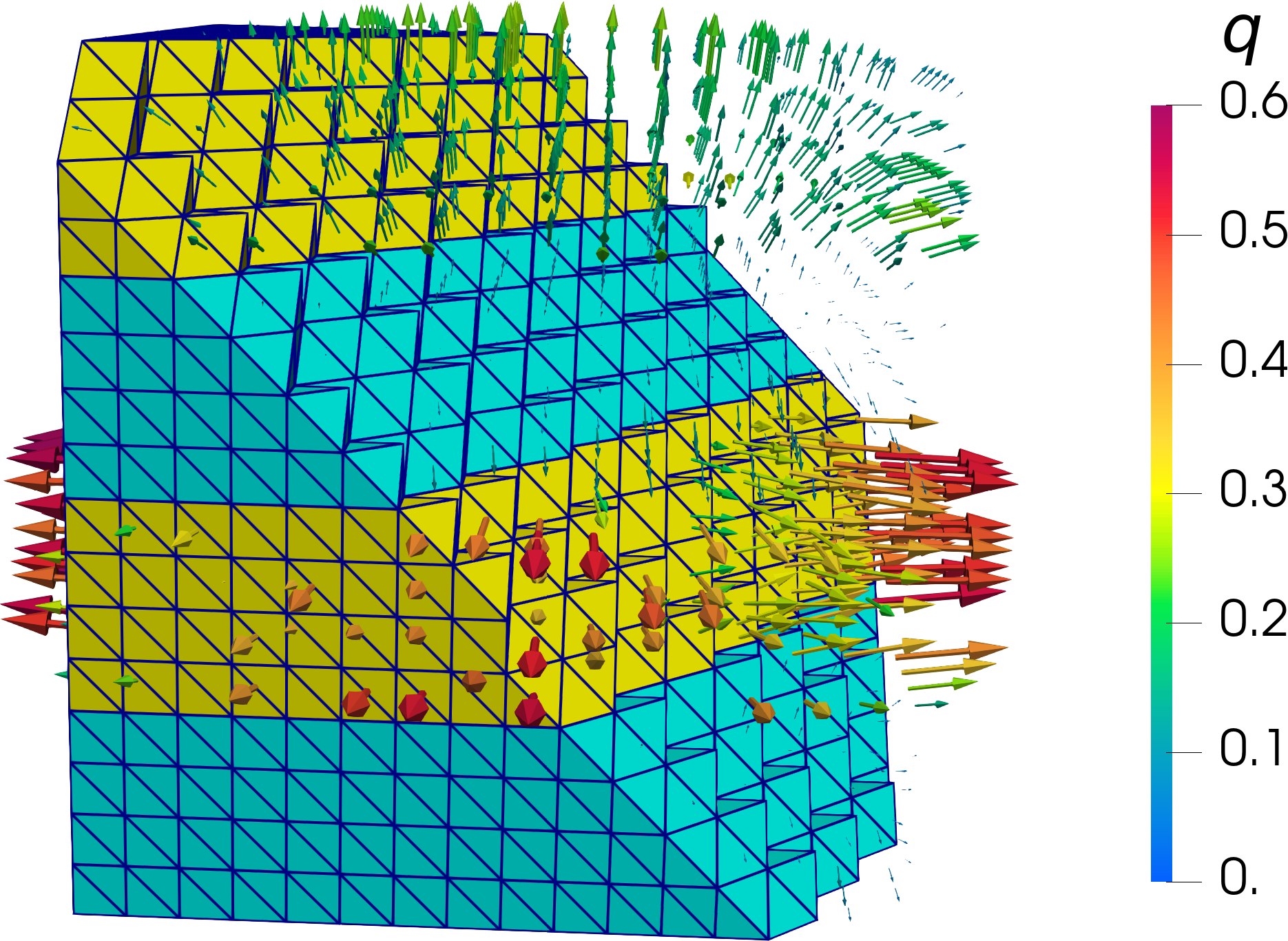}
    \caption{(left) The singular values decay rapidly so that only the first 19 singular values are larger than the threshold value $\varepsilon$. The corresponding modes form the reduced basis. (right) The reference solution of the flux superimposed on the layered conductivity field, with approximately half of the domain shown for the sake of visibility. The yellow subdomain is more permeable in this reference case.}
    \label{fig: case 1}
\end{figure}

The threshold value is set to $\varepsilon = 10^{-7}$ and we consider the parameter set $(\bar{K}, \bar{f}, \bar{\bm{\alpha}}) := (10^3, 1, \bm{0})$. The solution obtained from the three-step procedure with RBM is compared to the solution of the full-order model, illustrated in Figure~\ref{fig: case 1}(right). As shown in Table~\ref{tab: summary numerics}, we obtain a relative error on the order of $10^{-7}$ for both pressure and flux. Thus, the second step can be reduced to a system of merely 19 degrees of freedom and yield an accurate approximation of the solution.

\subsubsection{A two-dimensional fracture flow problem on a polygonal mesh}

Our second test case introduces two complexities, namely the incorporation of a fracture network and the use of a polygonal mesh, cf. Figure~\ref{fig: geometry case 2}. The former is taken care of using the mixed-dimensional differential operators from Section~\ref{sub: MiDi}. The latter is handled by discretizing with the mixed virtual element method of lowest order.

The geometry and material parameters are based on \cite[Case 3]{flemisch2018benchmarks}. A heterogeneity is introduced by letting two fractures, denoted $\Omega_f^{-}$, be blocking and letting the remaining eight fractures, $\Omega_f^{+}$, be conductive. The fracture network is illustrated in Figure~\ref{fig: geometry case 2}(left). We denote the surrounding bulk matrix by $\Omega_m$ and set the following conductivities:
\begin{align*}
    K(\bm{x}) &=
    \begin{cases}
        1, & \bm{x} \in \Omega_m,   \\
        K^{-}, & \bm{x} \in \Omega_f^{-},   \\
        K^{+}, & \bm{x} \in \Omega_f^{+},
    \end{cases} &
    &\begin{aligned}
        \\
        K^{-} &\in \left[10^{-5}, 10^{-3}\right], \\
        K^{+} &\in \left[10^{3}, 10^{5}\right].
    \end{aligned} 
\end{align*}
The source term $f$ is set to zero. Let the aperture of each fracture be $\epsilon = 10^{-4}$. The effective conductivities, i.e. $K_i$ and $K_{ij}$ in \eqref{eqs: MD Darcy tedious}, are set as follows. For each fracture, we set $K_i := \epsilon K$ internally and $K_{ij} := \frac{2}{\epsilon} K$ on the interface with the bulk matrix, with $K$ the fracture conductivity. Finally, at each intersection point, we set the conductivity $K_{ij} := \frac{2}{\epsilon} K^\pm$ with $K^{\pm}$ the harmonic average of the conductivities of the adjacent fractures.

\begin{figure}[htb]
    \centering
    \includegraphics[height=4.75cm]{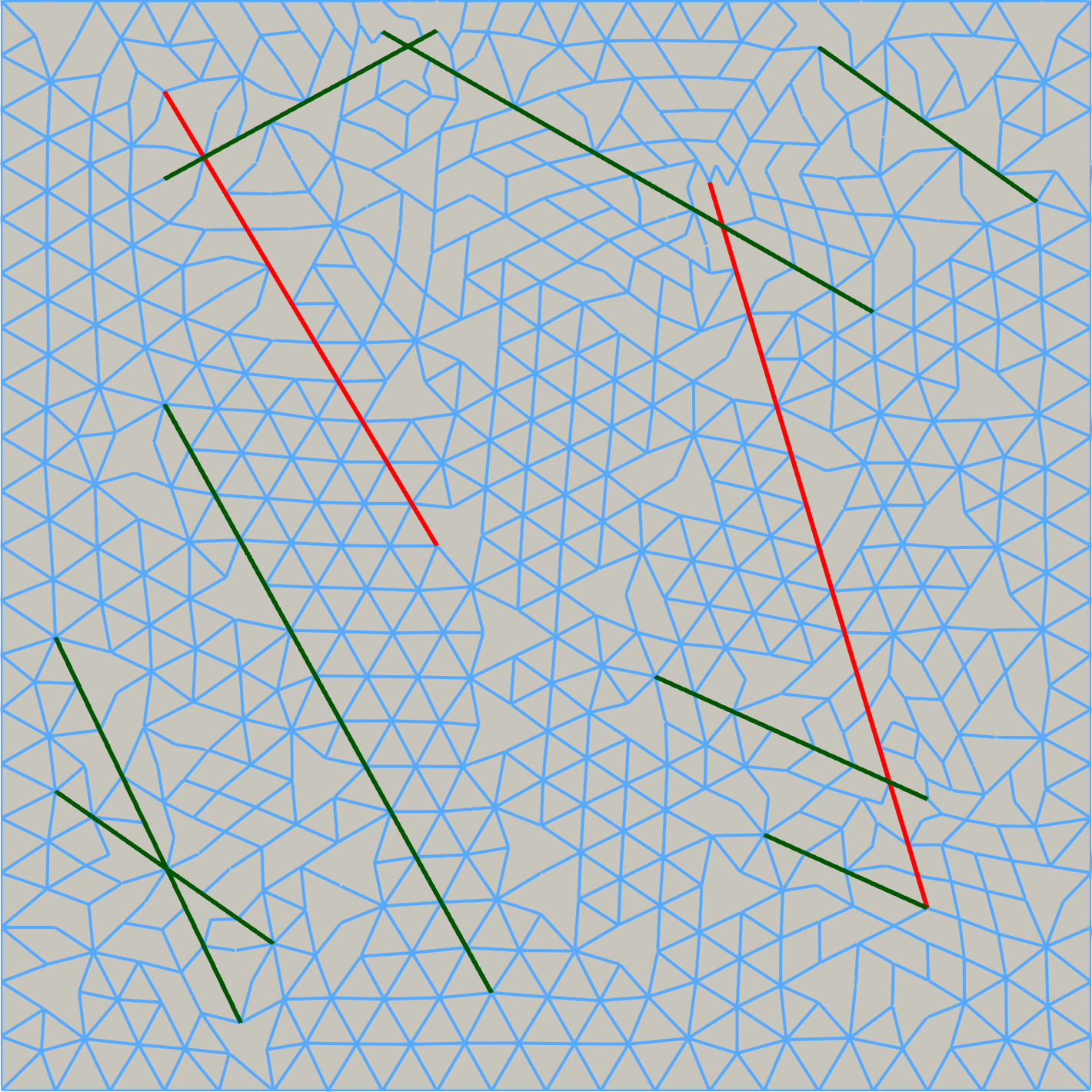}
    \
    \includegraphics[height=4.75cm]{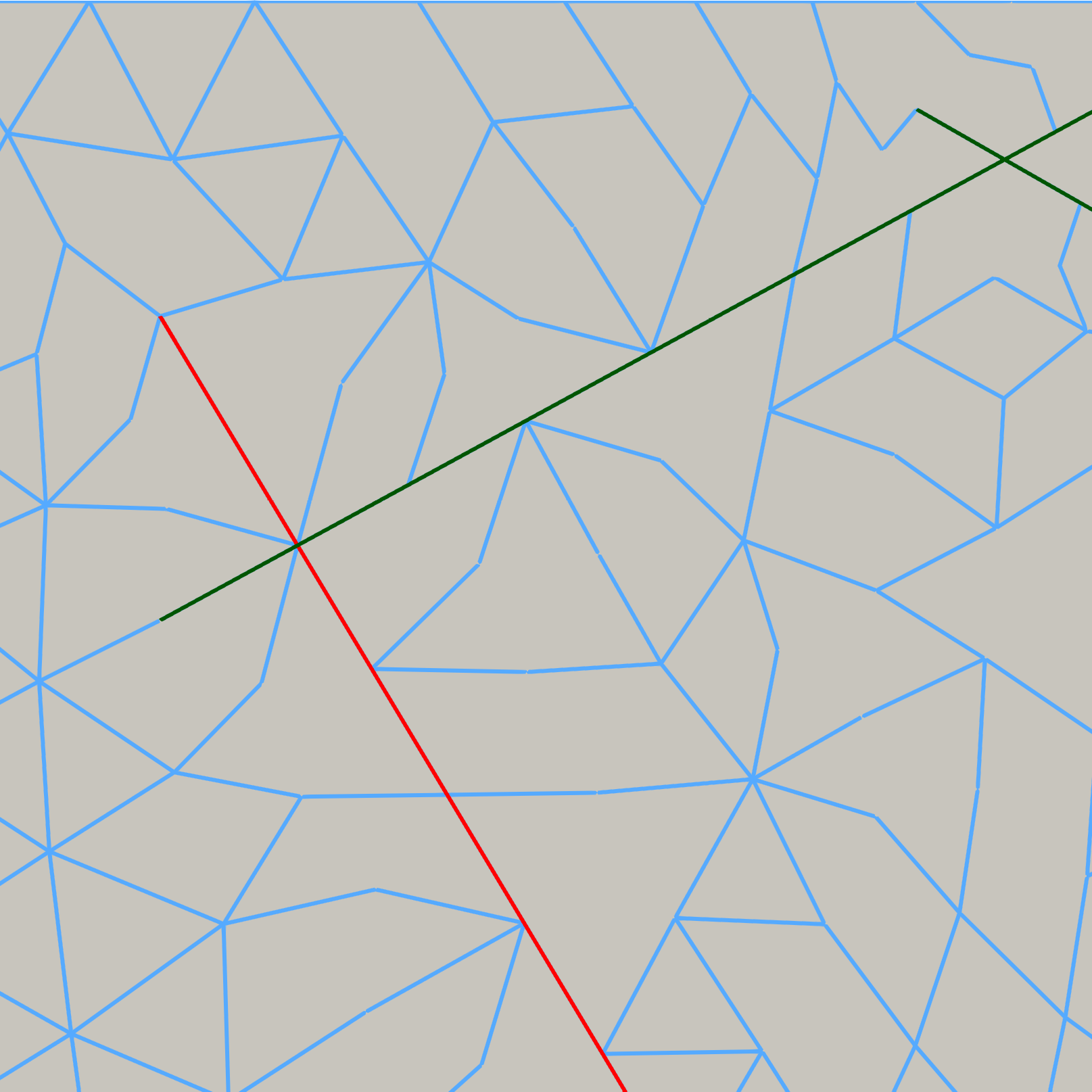}
    \caption{(left) The fracture network in the second test case with the conductive and blocking fractures highlighted in red and green, respectively. (right) A depiction of the top left of the domain showing the conforming, polygonal mesh on which the mixed virtual element spaces are defined.}
    \label{fig: geometry case 2}
\end{figure}

We emphasize that there is a no-flux condition at the fracture tips, cf. \eqref{eq: MD Darcy bcs}. These are essential boundary conditions in the mixed formulation and, in turn, we discretize the fracture flux variable on the subspace of functions that have zero normal trace at tips.

\begin{figure}[htb]
    \centering
    \includegraphics[height=4.75cm]{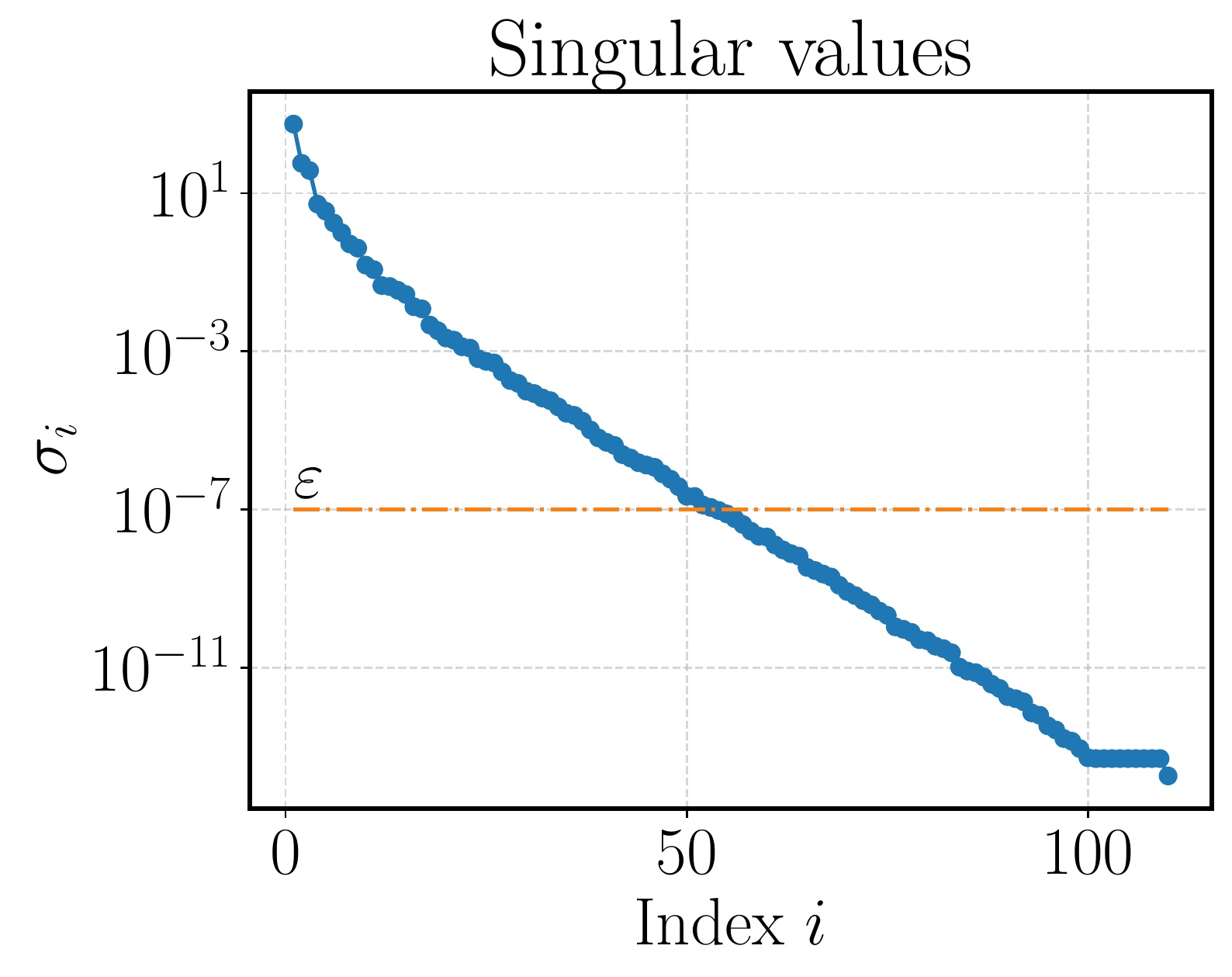}
    \
    \includegraphics[height=4.75cm]{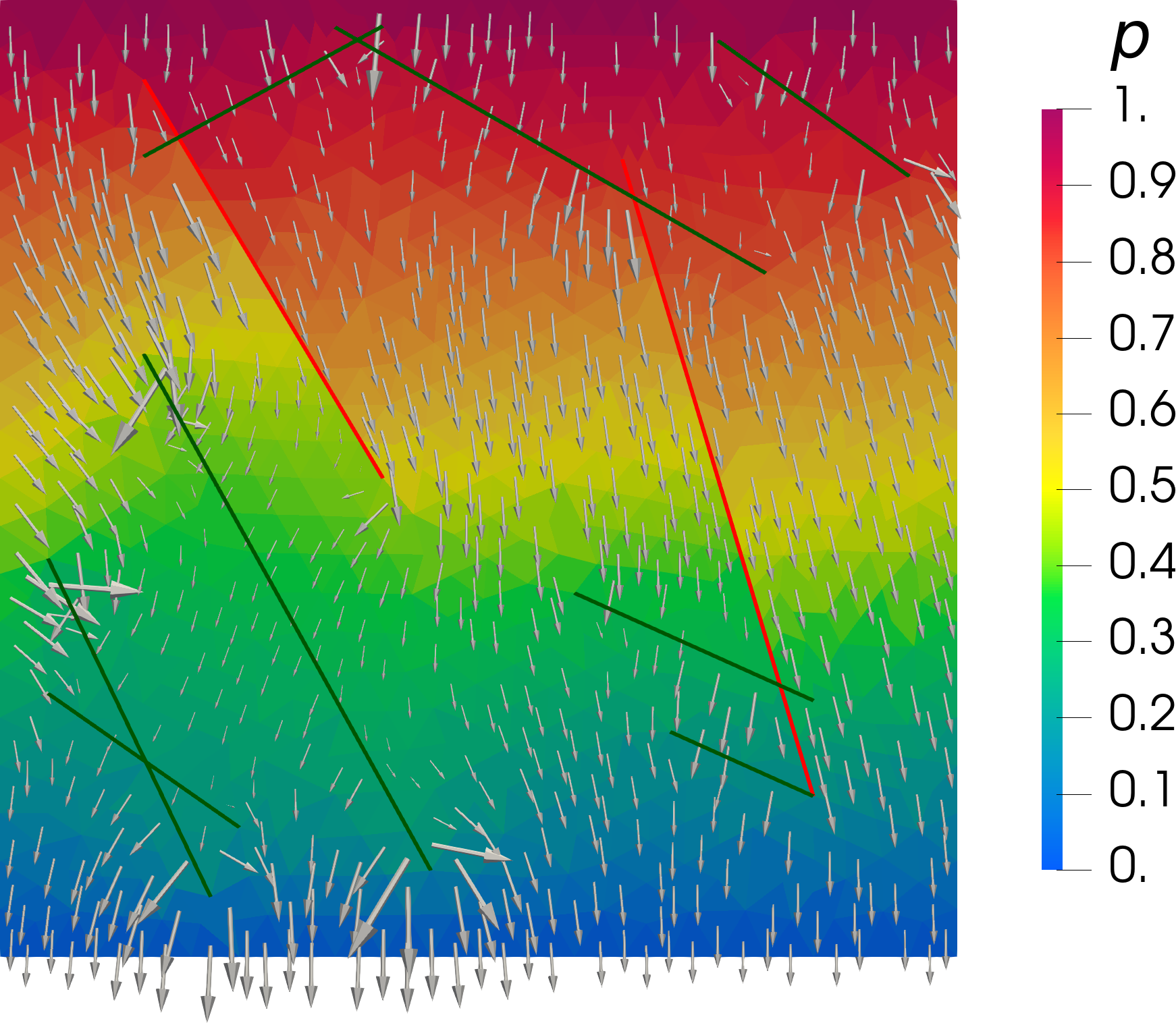}
    \caption{(left) The added complexity caused by the fracture network leads to a higher number of modes in the reduced basis. (right) The reference solution with the flux field superimposed on the pressure distribution. Note that the flow aligns itself tangentially with the blocking features and normally with the conductive fractures, as expected.}
    \label{fig: case 2}
\end{figure}

We compute 120 snapshots according to a Latin hypercube sampling of the parameters in which the conductivities are again log-uniformly distributed. The resulting singular values are presented in Figure~\ref{fig: case 2}(left). The decay of the singular values is significantly slower in this case compared to Figure~\ref{fig: case 1}, due to the added complexity caused by the fracture network. In fact, though the maximal value is approximately two orders lower than in the first case, it requires 53 modes to reach the threshold value of $\varepsilon = 10^{-7}$.

The parameters for the reference solution are $(K^-, K^+) = (10^{-4}, 10^4)$ as in \cite{flemisch2018benchmarks} and we set the boundary conditions with $\bar{\bm{\alpha}} = (0, 1)$ to induce a downward flow.

Finally, we remark on the number of degrees of freedom in the second step. The space on which the curl acts, $H \Lambda^{n - 2}$, is defined on the nodes of the two-dimensional mesh. This means that the solenoidal correction $(\mathfrak{D} \times \mathfrak{r})$ is completely dictated by degrees of freedom in the bulk, not the fractures or their intersections. Similarly, in 3D the space is $P \mathfrak{L}^1$ is only defined on manifolds of dimensions two and three, as is the case in the next numerical experiment.

\subsubsection{A three-dimensional fractured porous medium}

Our final test case is based on the regular fracture network of \cite[Case 4]{berre2021verification}. In order to incorporate fracture tips, we enlarge the domain from the unit cube to $[-0.1, 1.1]^3$. The network $\Omega_f$ is thereby fully immersed in the computational domain and we refer to the surrounding bulk as $\Omega_m$. Let us set the following parameters:
\begin{align*}
    K(\bm{x}) &=
    \begin{cases}
        1, & \bm{x} \in \Omega_m, \\
        \bar{K}, & \bm{x} \in \Omega_f,
    \end{cases} &
    &\begin{aligned}
        \\
        \bar{K} &\in \left[10^{3}, 10^{5}\right],
    \end{aligned} \\
    f(\bm{x}) &=
    \begin{cases}
        0, & \bm{x} \in \Omega_m,   \\
        \bar{f}, & \bm{x} \in \Omega_f,
    \end{cases} &
    &\begin{aligned}
        \\
        \bar{f} &\in \left[-1, 1\right].
    \end{aligned} 
\end{align*}

Thus, we consider a permeable fracture network on which we introduce a mass source. Again, the effect conductivities are obtained by scaling with the aperture $\epsilon = 10^{-4}$. In particular, we define $K_i := \epsilon^{n - d_i} \bar{K}$ on each lower-dimensional manifold $\Omega_i$ of dimension $d_i$ and $K_{ij} := \frac{2}{\epsilon} \bar{K}$ on each interface $\Gamma_{ij}$. As a result of integration in the normal directions, the effective source term becomes $f_i := \epsilon^{n - d_i} \bar{f}$ on each $\Omega_i$ with $d_i < n$.

\begin{figure}[htb]
    \centering
    \includegraphics[height=4.5cm]{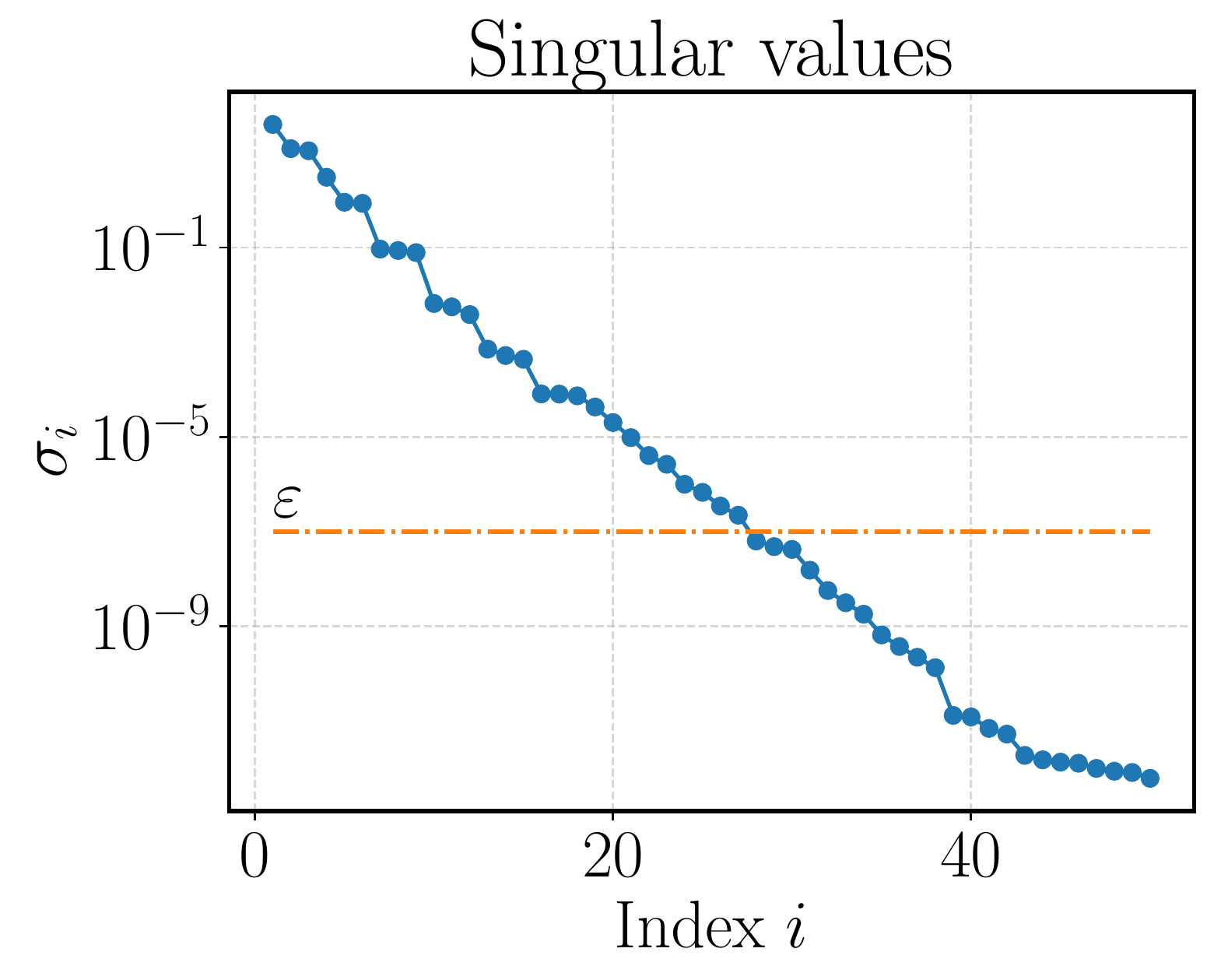}
    \
    \includegraphics[height=4.5cm]{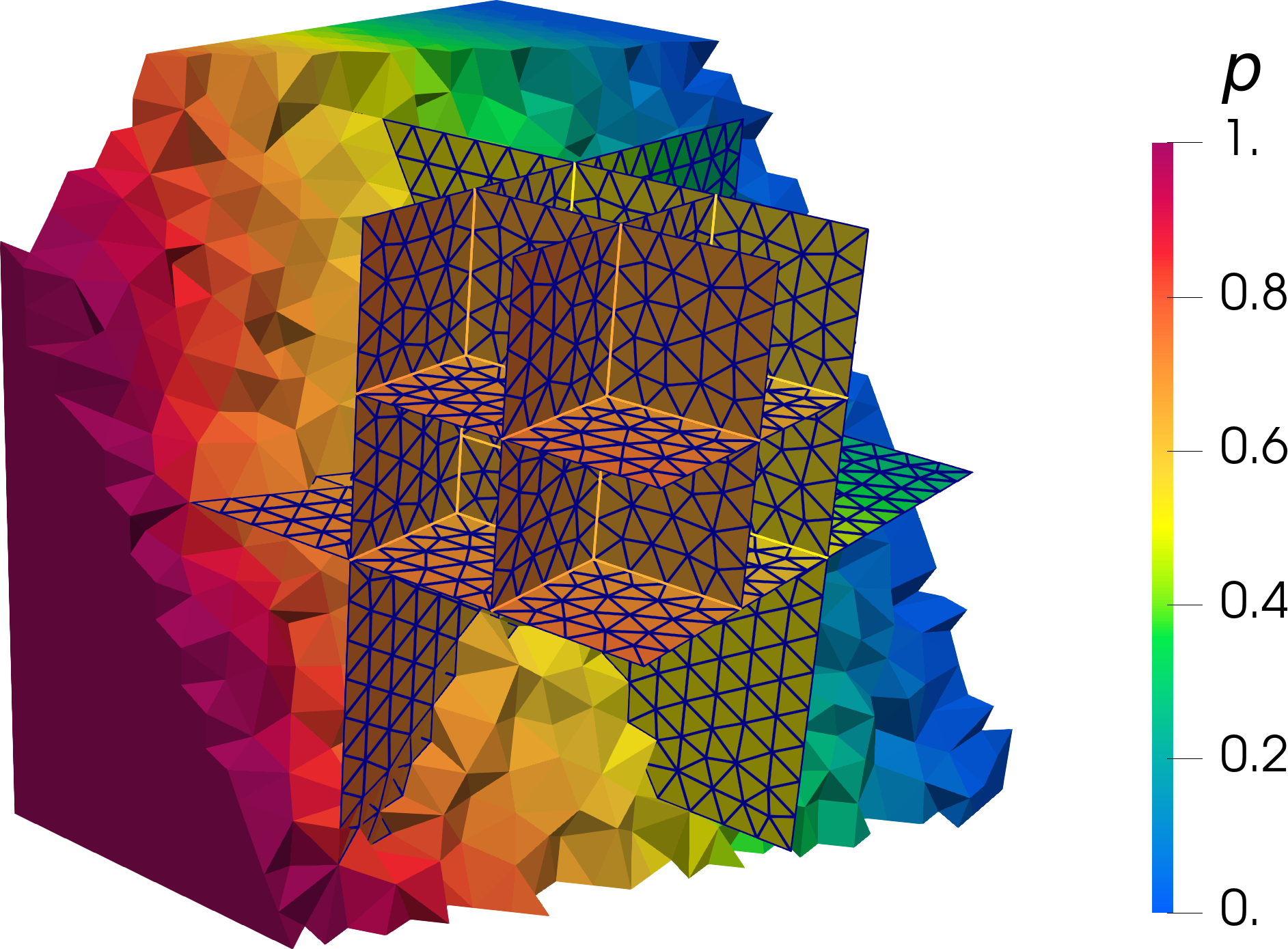}
    \caption{(left) The behavior of the solution to the four parameters of the third test case can be captured using 27 reduced basis functions. (right) The geometry consists of an immersed, regular fracture network in 3D. In the reference case, the flow is induced by a linear distribution on the boundaries and a source term in the network.}
    \label{fig: case 3}
\end{figure}

The reference solution for this case corresponds to the parameters $(\bar{K}, \bar{f}) = (10^4, 1)$ and $\bar{\bm{\alpha}} = (1, 0, 0)$. As shown in Figure~\ref{fig: case 3}(left), the reduced basis consists of 27 modes, with which we obtain an accurate description of the reference solution. More precisely Table~\ref{tab: summary numerics} shows the accuracy to be on the order of $10^{-8}$ and $10^{-10}$ for the flux and pressure, respectively.

This test case is the most computationally demanding of the three, with over 120k degrees of freedom in the original formulation. Following Section~\ref{sec: implementation}, the $LU$-decomposition of the cell-centered TPFA problem is saved in the \emph{off-line} stage and we emphasize that the RBM reduces the second step from 53,007 to 27 degrees of freedom. In this way, the computational cost for solving the reference problem decreases from approximately 6.6 minutes using a direct solver to 0.5 seconds in our implementation (speed-up factor $\approx800$).

\section{Concluding Remarks}
\label{sec: Conclusions}

We have proposed a three-step solution procedure for Darcy flow systems based on the exact de Rham complex. The mass conservation equation is first solved and we subsequently correct the flux field by adding a solenoidal vector field. We have shown how reduced basis methods can be used to relieve the computational cost in constructing the correction. Furthermore, the procedure was extended to the setting of Darcy flow in fractured porous media by employing mixed-dimensional differential operators.

The proposed procedure can be viewed from three perspectives. First, in the abstract setting (Section~\ref{sub: general}), it constructs the solution to the mixed formulation of a Laplace problem by solving three, related Hodge Laplace problems in primal form. Second, in the context of discretization methods (Section~\ref{sub: Discretization}), the procedure utilizes the efficiency of the TPFA finite volume method and applies a suitable correction to obtain the mixed finite (or virtual) element solution to the original problem. Third, from the algebraic perspective (Section~\ref{sec: implementation}), we can view the proposed procedure as approximating the Schur-complement of the original problem and applying a suitable correction.

Topics for future research will further explore different ways to relieve the computational effort in the second step, exploiting the fact that this does not influence the mass balance. Thus, the performance of approximate solvers such as Krylov subspace methods, multi-grid solvers or techniques based on deep learning will be investigated to approximate the vector potential.

\bibliographystyle{spmpsci}
\bibliography{references}

\begin{thebibliography}{10}
\providecommand{\url}[1]{{#1}}
\providecommand{\urlprefix}{URL }
\expandafter\ifx\csname urlstyle\endcsname\relax
  \providecommand{\doi}[1]{DOI~\discretionary{}{}{}#1}\else
  \providecommand{\doi}{DOI~\discretionary{}{}{}\begingroup
  \urlstyle{rm}\Url}\fi

\bibitem{arnold2018finite}
Arnold, D.N.: Finite element exterior calculus.
\newblock SIAM (2018)

\bibitem{arnold2006finite}
Arnold, D.N., Falk, R.S., Winther, R.: Finite element exterior calculus,
  homological techniques, and applications.
\newblock Acta numerica \textbf{15}, 1--155 (2006)

\bibitem{pygeon}
Ballini, E., Boon, W.M., Fumagalli, A., Scotti, A.: {PyGeoN}: A {P}ython
  package for {G}eo-{N}umerics (2022).
\newblock \urlprefix\url{https://github.com/compgeo-mox/pygeon}

\bibitem{baranger1996connection}
Baranger, J., Maitre, J.F., Oudin, F.: Connection between finite volume and
  mixed finite element methods.
\newblock ESAIM: Mathematical Modelling and Numerical Analysis \textbf{30}(4),
  445--465 (1996)

\bibitem{berre2021verification}
Berre, I., Boon, W.M., Flemisch, B., Fumagalli, A., Gl{\"a}ser, D.,
  Keilegavlen, E., Scotti, A., Stefansson, I., Tatomir, A., Brenner, K.,
  et~al.: Verification benchmarks for single-phase flow in three-dimensional
  fractured porous media.
\newblock Advances in Water Resources \textbf{147}, 103759 (2021)

\bibitem{boffi2013mixed}
Boffi, D., Brezzi, F., Fortin, M., et~al.: Mixed finite element methods and
  applications, vol.~44.
\newblock Springer (2013)

\bibitem{boon2020convergence}
Boon, W.M., Nordbotten, J.M.: Convergence of a {TPFA} finite volume scheme for
  mixed-dimensional flow problems.
\newblock In: International Conference on Finite Volumes for Complex
  Applications, pp. 435--444. Springer (2020)

\bibitem{boon2021stable}
Boon, W.M., Nordbotten, J.M.: Stable mixed finite elements for linear
  elasticity with thin inclusions.
\newblock Computational Geosciences \textbf{25}(2), 603--620 (2021)

\bibitem{boon2021functional}
Boon, W.M., Nordbotten, J.M., Vatne, J.E.: Functional analysis and exterior
  calculus on mixed-dimensional geometries.
\newblock Annali di Matematica Pura ed Applicata (1923-) \textbf{200}(2),
  757--789 (2021)

\bibitem{boon2018robust}
Boon, W.M., Nordbotten, J.M., Yotov, I.: Robust discretization of flow in
  fractured porous media.
\newblock SIAM Journal on Numerical Analysis \textbf{56}(4), 2203--2233 (2018)

\bibitem{budisa2020mixed}
Budi\v{s}a, A., Boon, W.M., Hu, X.: Mixed-dimensional auxiliary space
  preconditioners.
\newblock SIAM Journal on Scientific Computing \textbf{42}(5), A3367--A3396
  (2020)

\bibitem{falk2013stokes}
Falk, R.S., Neilan, M.: Stokes complexes and the construction of stable finite
  elements with pointwise mass conservation.
\newblock SIAM Journal on Numerical Analysis \textbf{51}(2), 1308--1326 (2013)

\bibitem{flemisch2018benchmarks}
Flemisch, B., Berre, I., Boon, W., Fumagalli, A., Schwenck, N., Scotti, A.,
  Stefansson, I., Tatomir, A.: Benchmarks for single-phase flow in fractured
  porous media.
\newblock Advances in Water Resources \textbf{111}, 239--258 (2018)

\bibitem{geuzaine2009gmsh}
Geuzaine, C., Remacle, J.F.: Gmsh: A 3-{D} finite element mesh generator with
  built-in pre-and post-processing facilities.
\newblock International journal for numerical methods in engineering
  \textbf{79}(11), 1309--1331 (2009)

\bibitem{hesthaven2016certified}
Hesthaven, J.S., Rozza, G., Stamm, B., et~al.: Certified reduced basis methods
  for parametrized partial differential equations, vol. 590.
\newblock Springer (2016)

\bibitem{hiptmair2007nodal}
Hiptmair, R., Xu, J.: Nodal auxiliary space preconditioning in h (curl) and h
  (div) spaces.
\newblock SIAM Journal on Numerical Analysis \textbf{45}(6), 2483--2509 (2007)

\bibitem{hu2022family}
Hu, K., Zhang, Q., Zhang, Z.: A family of finite element stokes complexes in
  three dimensions.
\newblock SIAM Journal on Numerical Analysis \textbf{60}(1), 222--243 (2022)

\bibitem{keilegavlen2021porepy}
Keilegavlen, E., Berge, R., Fumagalli, A., Starnoni, M., Stefansson, I.,
  Varela, J., Berre, I.: Porepy: An open-source software for simulation of
  multiphysics processes in fractured porous media.
\newblock Computational Geosciences \textbf{25}(1), 243--265 (2021)

\bibitem{licht2017complexes}
Licht, M.W.: Complexes of discrete distributional differential forms and their
  homology theory.
\newblock Foundations of Computational Mathematics \textbf{17}(4), 1085--1122
  (2017)

\bibitem{martin2005modeling}
Martin, V., Jaffr{\'e}, J., Roberts, J.E.: Modeling fractures and barriers as
  interfaces for flow in porous media.
\newblock SIAM Journal on Scientific Computing \textbf{26}(5), 1667--1691
  (2005)

\bibitem{nedelec1980mixed}
N{\'e}d{\'e}lec, J.C.: Mixed finite elements in {R}3.
\newblock Numerische Mathematik \textbf{35}(3), 315--341 (1980)

\bibitem{neilan2015discrete}
Neilan, M.: Discrete and conforming smooth de rham complexes in three
  dimensions.
\newblock Mathematics of Computation \textbf{84}(295), 2059--2081 (2015)

\bibitem{nordbotten2017modeling}
Nordbotten, J., Boon, W.: Modeling, structure and discretization of
  hierarchical mixed-dimensional partial differential equations.
\newblock In: International Conference on Domain Decomposition Methods, pp.
  87--101. Springer (2017)

\bibitem{quarteroni2015reduced}
Quarteroni, A., Manzoni, A., Negri, F.: Reduced basis methods for partial
  differential equations: an introduction, vol.~92.
\newblock Springer (2015)

\bibitem{quarteroni2014reduced}
Quarteroni, A., Rozza, G., et~al.: Reduced order methods for modeling and
  computational reduction, vol.~9.
\newblock Springer (2014)

\bibitem{raviart1977mixed}
Raviart, P.A., Thomas, J.M.: A mixed finite element method for 2-nd order
  elliptic problems.
\newblock In: Mathematical aspects of finite element methods, pp. 292--315.
  Springer (1977)

\bibitem{sirovich1987turbulence}
Sirovich, L.: Turbulence and the dynamics of coherent structures. i. coherent
  structures.
\newblock Quarterly of applied mathematics \textbf{45}(3), 561--571 (1987)

\bibitem{spivak2018calculus}
Spivak, M.: Calculus on manifolds: a modern approach to classical theorems of
  advanced calculus.
\newblock CRC press (2018)

\bibitem{beirao2013basic}
Beir{\~a}o~da Veiga, L., Brezzi, F., Cangiani, A., Manzini, G., Marini, L.D.,
  Russo, A.: Basic principles of virtual element methods.
\newblock Mathematical Models and Methods in Applied Sciences \textbf{23}(01),
  199--214 (2013)

\end{thebibliography}

\section*{Declarations}

This project has received funding from the European Union's Horizon 2020 research and innovation programme under the Marie Skłodowska-Curie grant agreement No. 101031434 -- MiDiROM.
The authors have no relevant financial or non-financial interests to disclose.
The datasets and source code generated and analyzed during the current study are available in the repository \url{https://github.com/compgeo-mox}.

\end{document}